\documentclass[12pt]{article}
\usepackage{pifont}
\usepackage{amsmath}
\usepackage{amssymb,float}
\usepackage{amsthm,graphicx,pifont}
\usepackage{graphicx,psfrag,float,subfigure}
\usepackage{caption,enumerate}
\usepackage[numbers,sort&compress]{natbib}
\usepackage[scale=0.8,a4paper]{geometry}
\usepackage[colorlinks]{hyperref}
\usepackage{amsthm,amsmath,amssymb}
\usepackage{mathrsfs}
\newtheorem {claim}{Claim}
\newtheorem{theorem}{Theorem}[section]

\newtheorem{lemma}[theorem]{Lemma}

\newtheorem{pro}[theorem]{Problem}

\title{Solid bricks that every $b$-invariant edge is solitary}
\author{{\small\bf Yipei Zhang$^{1}$, Xiumei Wang$^{2}$}\thanks{Corresponding author. Email address: wangxiumei@zzu.edu.cn}\\
{\small $^{1}$School of Mathematics and Statistics, North China University of Water Resources and Electric Power,}\\
{\small Zhengzhou, Henan 450046,  China}\\
{\small $^{2}$School of Mathematics and Statistics, Zhengzhou University,}\\
{\small Zhengzhou, Henan 450001,  China}}
\date{}

\makeatother
\begin{document}
\maketitle

\begin{abstract}
A graph $G$ is a brick if it is 3-connected and $G-\{u,v\}$ has a perfect matching for any two distinct vertices $u$ and $v$ of $G$.
A brick $G$ is {\it solid} if for any two vertex disjoint odd cycles $C_1$ and $C_2$ of $G$,
$G-(V(C_1)\cup V(C_2))$ has no perfect matching.
Lucchesi and Murty proposed a problem concerning the characterization of bricks, distinct from $K_4$, $\overline{C_6}$ and the Petersen graph, in which every $b$-invariant edge is solitary.
In this paper, we show that for a solid brick $G$ of order $n$ that is distinct from $K_4$, every $b$-invariant edge of $G$ is solitary if and only if $G$ is a wheel $W_n$.\\

\end{abstract}

\section{Introduction}

All  graphs considered in this paper are finite and simple.
For any undefined notation and terminology, we follow \cite{BM08}.
Let $G$ be a graph with vertex set $V(G)$ and edge set $E(G)$.
The number of vertices in $G$ is referred to as its {\it order}.
The complete graph of order $n$ is denoted by $K_n$.
The {\it length} of a path or a cycle is defined as the number of its edges.
A path or cycle is {\it odd} or {\it even} according to the parity of its length.

For two nonempty proper subsets $X$ and $Y$ of $V(G)$, we use $E[X,Y]$ to denote the set of all the edges of $G$ with one end in $X$ and the other end
in $Y$.
For a nonempty proper subset $X$ of $V(G)$, we use $N(X)$ to denote the set of all the vertices  in $\overline{X}$ that have one neighbour in $X$, and use $\partial(X)$ to denote the set $E[X,\overline{X}]$, where $\overline{X}=V(G)\setminus X$.
The set $\partial(X)$ is referred to as a {\it cut} of $G$.
A cut $\partial(X)$ of $G$ is {\it trivial} if one of $X$ and $\overline{X}$ has precisely one vertex,
and is \emph{nontrivial} otherwise.
A cut $\partial(X)$ of $G$ is {\it tight} if each perfect matching of $G$ contains precisely one edge in $\partial(X)$.
For a cut $\partial(X)$ of $G$, we denote by
$G/X$ and $G/\overline{X}$ the two graphs obtained from $G$ by shrinking $X$ and $\overline{X}$ to a single vertex, respectively, and these two graphs are called the {\it $\partial(X)$-contractions} of $G$.

A connected graph $G$ of order at least two is \emph{matching covered} if every edge of $G$ is contained in a perfect matching of $G$.
A matching covered graph that is free of nontrivial tight cuts is
called a \emph{brick} if it is nonbipartite, and  a \emph{brace} otherwise.
Noting that if a matching covered graph $G$ has a nontrivial tight cut $\partial(X)$, then the two $\partial(X)$-contractions of $G$ are also matching covered.
By repeatedly performing such contractions until no nontrivial tight cuts remain, we obtain a list of graphs that are either bricks or braces.
This procedure is known as a {\it tight cut decomposition} of $G$.
Since the choice of  nontrivial tight cut at each  step may be different,
a matching covered graph may admit several tight cut decompositions.
However, Lov\'asz \cite{Lovasz1987} showed that any two tight cut decompositions of a matching covered graph yield the same list of bricks and braces (up to multiple edges).
These bricks are referred to as  the bricks of $G$, and their number, denoted by $b(G)$,  is well-defined. Note that $b(G)=1$ when $G$ is a brick.

Edmonds et al. \cite{ELP82} gave an equivalent definition of bricks by showing that a graph $G$ is a brick if and only if it is 3-connected and $G-\{u,v\}$ has a perfect matching for any two distinct vertices $u$ and $v$ of $G$. Thus, every brick has at least four vertices and its  minimum degree  is at least three.

\begin{figure}[h]
 \centering
 \includegraphics[width=0.8\textwidth]{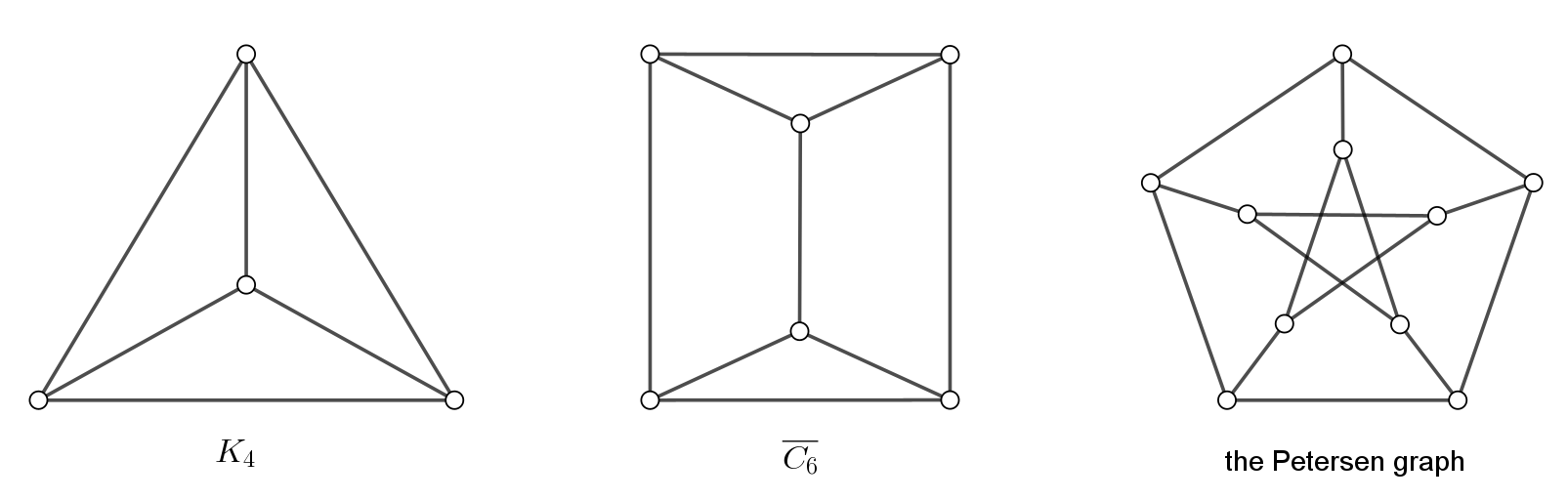}\\
 \caption{Three bricks.}\label{fig1}
\end{figure}

\begin{figure}[h]
 \centering
 \includegraphics[width=\textwidth]{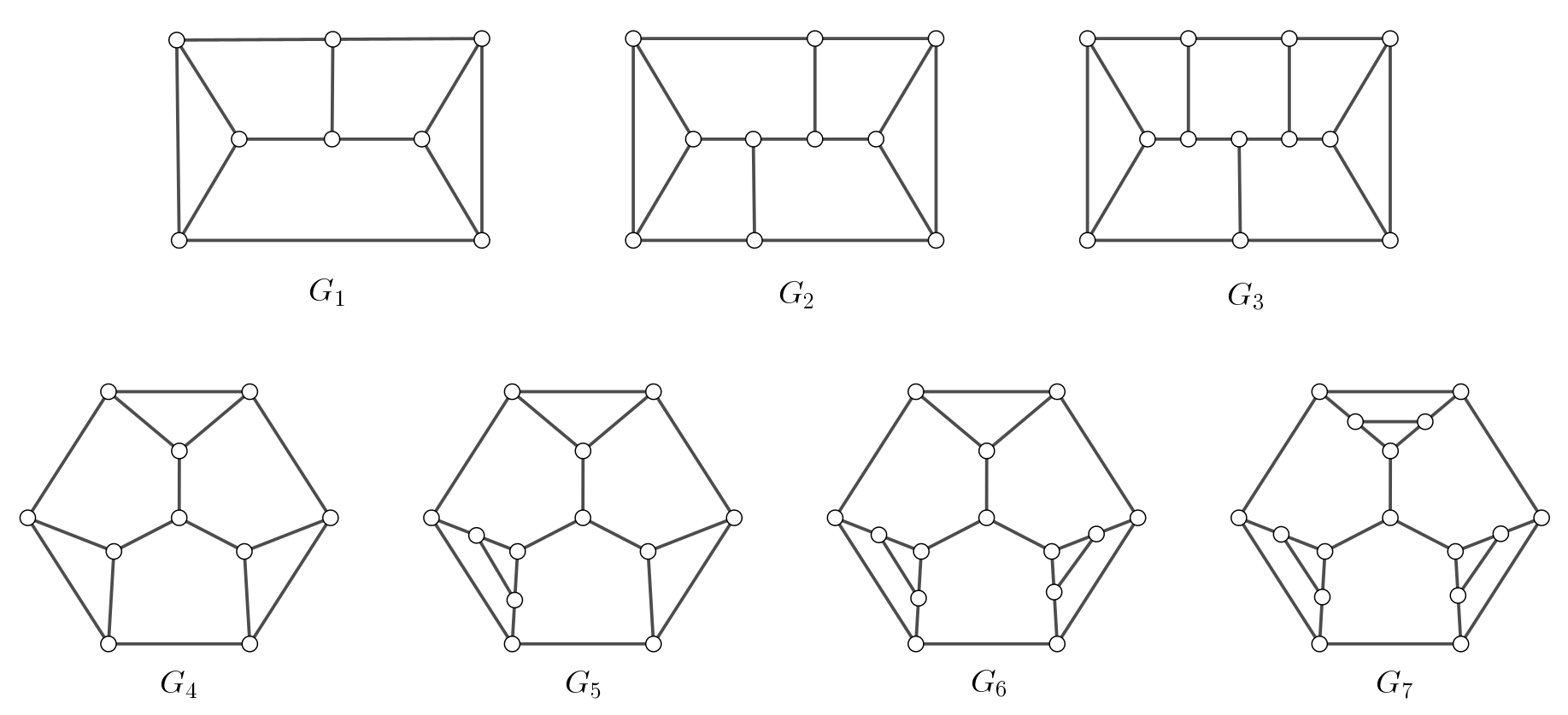}\\
 \caption{The family $\mathcal{G}$.}\label{fig2}
\end{figure}

Let $G$ be a matching covered graph.
An edge $e$ of $G$ is \emph{removable} if $G-e$ is also matching covered, and is {\it nonremovable} otherwise.
A removable edge $e$ of $G$ is \emph{$b$-invariant} if $b(G-e)=b(G)$.
For the existence of $b$-invariant edges in brick, Lov\'asz \cite{Lovasz1987} proposed the
conjecture that every brick different from $K_4$, $\overline{C_6}$ and the Petersen graph
(as shown in Figure \ref{fig1}) has a $b$-invariant edge.
Carvalho et al. \cite{CLM2002} confirmed this conjecture and  extended the result to include two $b$-invariant edges, with the exception of the graph $G_1$ in the family $\mathcal{G}$ (see Figure \ref{fig2}).
Kothari et al. \cite{KCLL2020} proved that every essentially 4-edge-connected cubic non-near-bipartite brick $G$, other than the Petersen graph, has at least $|V(G)|$ $b$-invariant edges.
They further conjectured that every essentially 4-edge-connected cubic near-bipartite bricks $G$,
other than $K_4$, has at least $\frac{|V(G)|}{2}$ $b$-invariant edges.
Lu et al. \cite{LFW2020} confirmed this conjecture and characterized the extremal graphs attaining this lower bound, which are prisms of order $4k+2$ and M\"{o}bius ladder of order $4k$ (where $k\geq2$).

An edge of a graph is {\it solitary} if it is contained in precisely one perfect matching of the graph, and is {\it nonsolitary} otherwise.
Recently, Lucchesi and Murty proposed the following problem (see Unsolved Problems 1 in \cite{LM2024}).

\begin{pro}[\cite{LM2024}]\label{pro}
Characterize bricks, distinct from $K_4$, $\overline{C_6}$ and the Petersen graph, in
which every $b$-invariant edge is solitary.
\end{pro}

For the Problem \ref{pro}, Zhang et al. \cite{ZLZ2025+} characterized the case when bricks are cubic, and  discovered the graph family $\mathcal{G}$, as shown in Figure \ref{fig2}.

\begin{theorem}[\cite{ZLZ2025+}]\label{cubic-brick}
Let $G$ be a cubic brick other than $K_4$, $\overline{C_6}$ and the Petersen graph. Every $b$-invariant edge of $G$ is solitary if and only if $G$ is one graph in $\mathcal{G}$.
\end{theorem}

A brick $G$ is {\it solid} if for any two vertex disjoint odd cycles $C_1$ and $C_2$ of $G$,
$G-(V(C_1)\cup V(C_2))$ has no perfect matching,  and is {\it nonsolid} otherwise.
By the definition, we can easily see that each graph of $\mathcal{G}$ is nonsolid, and
wheels with even order defined below are solid bricks.
A {\it wheel} $W_n$ of order $n$ is the graph obtained from a cycle $C$ of length $n-1$ by
adding a new vertex and connecting this vertex to each vertex of the cycle $C$.
The cycle is called the {\it rim}, the new vertex the {\it hub}, and each edge joining
the hub to the rim a {\it spoke}. Note that $n\geq4$.
When $n=4$, $W_n$ is $K_4$, which has no removable edge.
When $n\geq6$ and $n$ is even, it is straightforward to verify that
each spoke of $W_n$ is removable and all the other edges are nonremovable.
In this paper, we characterized the case in Problem \ref{pro} when bricks are solid.
The following is our main result.

\begin{theorem}\label{main-theorem}
Let $G$ be a solid brick of order $n$ distinct from $K_4$. Then every $b$-invariant edge of $G$ is solitary if and only if $G$ is a wheel $W_n$.
\end{theorem}

We organize the rest of the paper as follows. In Section \ref{section 2}, we present some notion and lemmas. In Section \ref{section 3}, we give a proof of Theorem \ref{main-theorem}.

\section{Preliminaries}\label{section 2}

Let $F$ and $G$ be two graphs. If $F$ is a subgraph of $G$, we denote it by $F\subseteq G$; otherwise, denote it by $F\varsubsetneq G$.
Let $M$ be a matching of $G$.
An {\it $M$-alternating path} or {\it cycle} of $G$ is a path or cycle whose edges are alternately in $M$ and $E(G)\backslash M$.
Note that an $M$-alternating path may or may not start or end with edges of $M$.
In particular, an $M$-alternating path neither starts nor ends with an edge of $M$ is called an \emph{open $M$-alternating path}. For instance, each edge of $E(G)\backslash M$ is an open $M$-alternating path. Clearly, the length of an open $M$-alternating path is odd.
For convenience, if $C$ is an odd cycle of $G$ such that $M\cap E(C)$ is the (unique) perfect matching of $C-u$, where $u\in V(C)$, then we say that the vertex $u$ is reachable from itself by an open $M$-alternating path.

\begin{lemma} [\cite{CLM2002}]\label{solid-removable-b}
In a solid brick, every removable edge is also b-invariant.
\end{lemma}

\begin{lemma} [\cite{CLM2012}]\label{solid-two-nonremovable}
If $G$ is a solid brick with at least six vertices, then for every vertex $v$ of $G$, at most two edges
incident with $v$ are nonremovable in $G$.
\end{lemma}

By Lemmas \ref{solid-removable-b} and \ref{solid-two-nonremovable}, one can easily get the following lemma.

\begin{lemma} \label{2-nonsolitary}
Let $G$ be a solid brick with at least six vertices and let $v$ be a vertex of $G$. If every $b$-invariant edge of $G$ is solitary, then at most two edges incident with $v$ are nonsolitary in $G$.
\end{lemma}

\begin{lemma} \label{Wn-b-solitary}
Every $b$-invariant edge of $W_n$ is solitary, where $n$ is even.
\end{lemma}

\begin{proof}
Suppose that $e$ is a $b$-invariant edge of $W_n$. By the definition, $e$ is removable in $W_n$.
Recall that each spoke of $W_n$ is removable in $W_n$, while all the  other edges are nonremovable.
So $e$ is a spoke of $W_n$.
Note that the graph obtained from $W_n$ by deleting the two ends of $e$ is an odd path, and so it has a unique perfect matching.
Therefore,  $e$ is solitary in $W_n$. The result holds.
\end{proof}

\section{Proof of Theorem \ref{main-theorem}}\label{section 3}

In this section, unless otherwise specified, we use $N(X)$ and $E[X,Y]$  instead of $N_G(X)$ and $E_G[X,Y]$, respectively, where $X,Y\subseteq V(G)$. Next, we begin to prove Theorem \ref{main-theorem}.

Let $G$ be a solid brick of order $n$ other than $K_4$.
Then $G$ is 3-connected and $\delta(G)\geq3$. Moreover, $n\geq4$ and $n$ is even.
If $n=4$, then $G$ is $K_4$, a contradiction.
So  $n\geq6$.
We first prove the sufficiency.  Suppose that $G$ is a wheel $W_n$.
By Lemma \ref{Wn-b-solitary}, every $b$-invariant edge of $W_n$ is solitary.
The result holds.

Next, we prove the necessity.  Suppose that every $b$-invariant edge of $G$ is solitary.
We assert that $G$ has a vertex of degree at least four.
Otherwise, $G$ is cubic because $\delta(G)\geq3$.
By the definition of solid bricks, $G$ is different from $\overline{C_6}$ and the Petersen graph.
Since every $b$-invariant edge of $G$ is solitary, Theorem \ref{cubic-brick} implies that
$G$ is one graph of $\mathcal{G}$.
Recall that each graph of $\mathcal{G}$ is nonsolid. So $G$ is nonsolid, a contradiction.
The assertion holds.

Assume that $u$ is a vertex of $G$ with degree at least four.
By Lemma \ref{2-nonsolitary}, at most two edges incident with $u$ are nonsolitary in $G$. This implies that at least two edges incident with $u$ are solitary in $G$ since $d(u)\geq4$.
Assume, without loss of generality,  that  $uu_1$ and $uu_2$ are  two such solitary edges.
Then each of $uu_1$ and $uu_2$ lies in a unique perfect matching of $G$.
Thus, $G-\{u,u_i\}$ has a unique perfect matching, say $M_i$, $i=1,2$.

Now we consider the edge-induced subgraph $P=G[M_1\Delta M_2]$, where $M_1\Delta M_2$ is the symmetric difference of $M_1$ and $M_2$.
If $P$ contains a cycle $C_1$, then $C_1$ is an $M_1$-alternating cycle of $G$,
which implies that $M_1\Delta C_1$ is also a perfect matching of $G-\{u,u_1\}$, a contradiction.
Thus, $P$ has  no cycle. Since $u_1$ and $u_2$ are the only  vertices of degree 1 in $P$,   $P$   is a path with edges alternate between $M_1$ and $M_2$.
Moreover, $u_i$ is an end of $P$, which is not covered by $M_i$, $i=1,2$.
Hence, $P$ is even and $|E(P)|\geq2$.
For convenience, we assume that  $P=v_1v_2\ldots v_t$ with  $v_1=u_1$ and $v_t=u_2$. 
Then $v_1v_2\in M_2$, $v_{t-1}v_t\in M_1$, $t$ is odd and $t\geq3$.
Let $V_1=\{v_1,v_3,\ldots,v_t\}$ and $V_2=\{v_2,v_4,\ldots,v_{t-1}\}$.

Recall that $M_1$ is the unique perfect matching of $G-\{u,v_1\}$ and $M_2$ is the unique perfect matching of $G-\{u,v_t\}$. Then the following lemma  holds.

\begin{lemma}\label{lemma31}
$G-\{u,v_1\}$ has no $M_1$-alternating cycle and $G-\{u,v_t\}$ has no $M_2$-alternating cycle.
\end{lemma}

Let $H=G-(V(P)\cup\{u\})$ and $M_1'=M_1\cap E(H)$. Then $M_1'=M_1\backslash E(P)=M_2\backslash E(P)=M_1\cap M_2$ is the unique perfect matching of $H$.

\begin{lemma}\label{lemma32}
Let $v_j\in V_1$ and $v_k\in V_2$. Then $v_j$ and $v_k$ cannot be connected by an open $M_1'$-alternating path in $G$ unless the path is $v_jv_k$ that is an edge  in $E(P)$.
In particular, $E[V_1,V_2]\backslash E(P)=\emptyset$.
\end{lemma}

\begin{proof}
Suppose to the contrary that $v_j$ and $v_k$ is connected by an open $M_1'$-alternating path $P'$ in $G$ and  $P'$ is not an edge  in $E(P)$.
Then  $P'$ is odd.
If $|E(P')|=1$, then $E(P')\cap E(P)=\emptyset$.
If $|E(P')|\neq1$, then $|E(P')|\geq3$ and  all the internal vertices of $P'$ are covered by $M_1'$. 
Hence, all the internal vertices of $P'$ lie in the graph $H$, which implies that $E(P')\cap E(P)=\emptyset$.
In short, whether $|E(P')|=1$ or not, we have $E(P')\cap E(P)=\emptyset$.

Since $v_j\in V_1$ and $v_k\in V_2$, we have $j\neq k$.
Let $C_1=v_jPv_kP'v_j$.
If $j<k$, then $C_1$ is an $M_2$-alternating cycle of $G-\{u,v_t\}$;
otherwise, $C_1$ is an $M_1$-alternating cycle of $G-\{u,v_1\}$.
In both cases, we  get a contradiction to Lemma \ref{lemma31}.
Hence, $v_j$ and $v_k$ cannot be connected by an open $M_1'$-alternating path in $G$
unless the path is  $v_jv_k$ that is an edge in  $E(P)$.
In particular, if $j$ and $k$ are not consecutive positive integers, then $v_jv_k\notin E(G)$.
This implies that  $E[V_1,V_2]\backslash E(P)=\emptyset$.
The result holds.
\end{proof}

Recall that each vertex of  $H$  is covered by an edge in $M_1'$. Thus, for any vertex $x_1$ in $H$,  there exists an $M_1$-alternating path in $H$ which starts at $x_1$ and starts  with an edge in $M_1$.
In the following lemma, we present a property with respect to such  maximal $M_1$-alternating paths.

\begin{lemma}\label{lemma33}
For any vertex $x$ in $P$, assume that $x_1$ is a neighbour of $x$ in $H$ and $P'=x_1x_2\cdots x_l$ is
a maximal $M_1'$-alternating path in $H$ starting at $x_1$  such that $x_1x_2\in M_1'$.
Then $l$ is even and $N(x_l)\cap V(P')\subseteq \{x_2,x_4,\ldots,x_{l-2},x_{l-1}\}$.
\end{lemma}

\begin{proof}
Recall that $M_1'$ is the unique perfect matching of $H$.
Since $P'$ is a maximal $M_1'$-alternating path in $H$ starting at $x_1$ such that $x_1x_2\in M_1'$,
we have $l\geq2$ and $x_{l-1}x_l\in M_1'$. So  $l$ is even.
Assume that $x_l$ has a neighbour $x_j$ in $P'$.
If $j\in\{1,3,\ldots,l-3\}$, then $x_jP'x_lx_j$ is an $M_1$-alternating cycle of $G-\{u,v_1\}$,  contradicting Lemma \ref{lemma31}.
Therefore, $j\in \{2,4,\ldots,l-2,l-1\}$. The result holds.
\end{proof}

\begin{figure}[h]
 \centering
 \includegraphics[width=\textwidth]{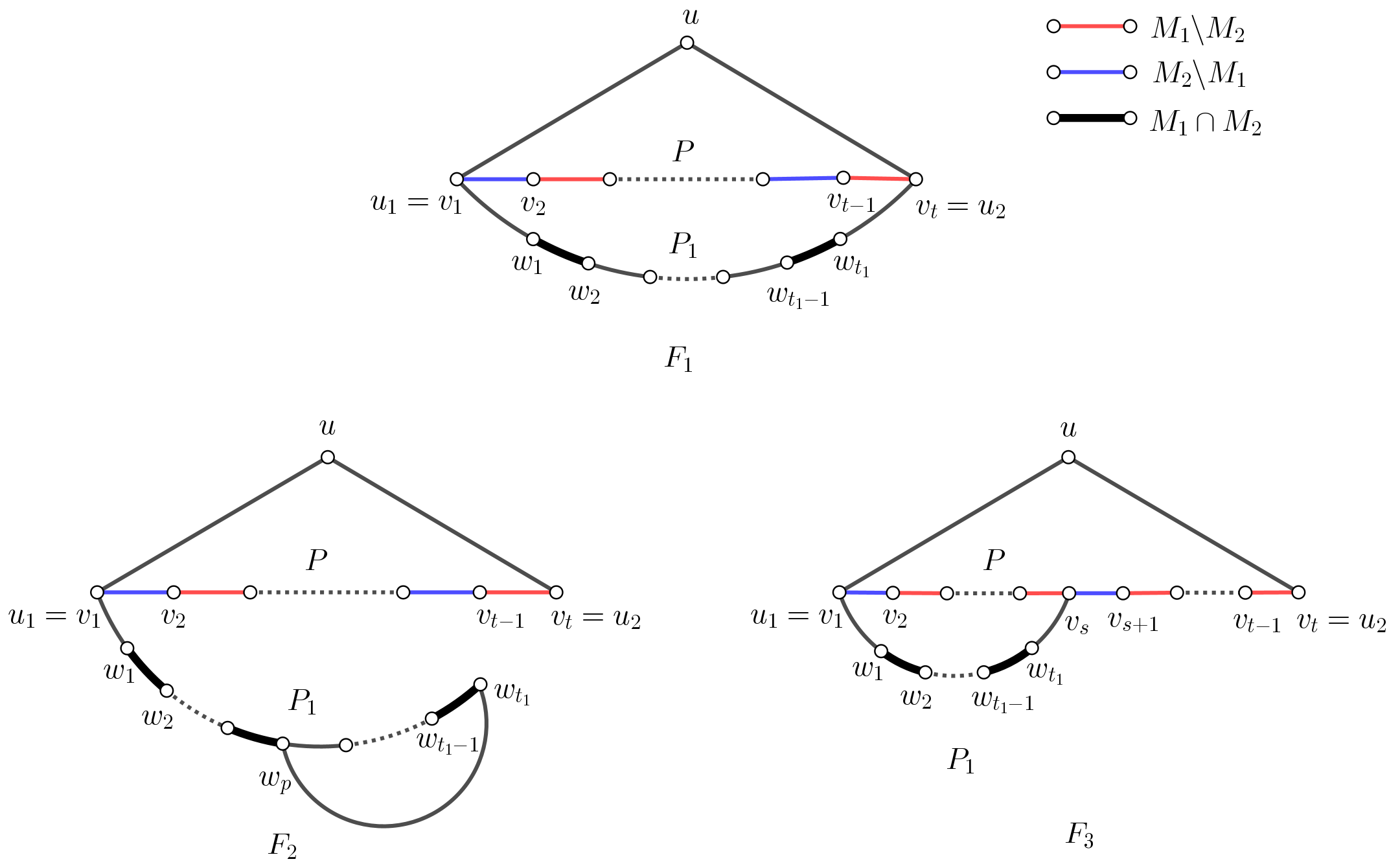}\\
 \caption{The three graphs.}\label{fig3}
\end{figure}

Note that an $M_1'$-alternating path  is also an $M_1$-alternating path.
Recall that $\delta(G)\geq3$. Then the vertex $v_1$ must have a neighbour in $H$ or $P-\{v_1,v_2\}$.
If $N(v_1)\cap V(H)=\emptyset$, then Lemma \ref{lemma32} implies that $\emptyset\neq (N(v_1)\cap V(P))\backslash\{v_2\}\subseteq V_1$.
If $v_1$ has a neighbour $w_1$ in $H$,  let $P_1=w_1w_2\cdots w_{t_1}$ be a maximal $M_1$-alternating path in $H$ starting at $w_1$  such that $w_1w_2\in M_1$.
Then $t_1\geq2$, $t_1$ is even and $w_{t_1-1}w_{t_1}\in M_1$.
Moreover, $M_1\cap E(P_1)$ is the unique perfect matching of $P_1$.
Let $F_1$, $F_2$ and $F_3$ be the three graphs as shown in Figure \ref{fig3},
where $u$, $P$ and $P_1$ are  defined as previously.
In the graph $F_2$, for the vertex $w_p$, $p$ is even and $2\leq p\leq t_1-2$;
in the graph $F_3$, for the vertex $v_s$, $s$ is odd and $1\leq s\leq t-2$.
Note that   in the graph $F_1$ it is possible that the path $v_1w_1w_2\cdots w_{t_1}v_t$ is an edge $v_1v_t$, and  in the graph $F_3$ it is possible that the path $v_1w_1w_2\cdots w_{t_1}v_s$ is an edge $v_1v_s$.
We can see that $F_1$, $F_2$ and $F_3$ are possible  subgraphs of $G$. The following lemma shows that  only $F_1$  can occur as a subgraph of $G$.

Let $H_1=G-(V(P)\cup V(P_1)\cup\{u\})$. Then $H_1=H-V(P_1)$.
Since  $M_1' ~(=M_1\cap E(H))$ and $M_1\cap E(P_1)$ are the unique perfect matchings of $H$ and $P_1$, respectively, $M_1\cap E(H_1)$ is the unique perfect matching of $H_1$. In Lemma \ref{lemma33}, if we replace $H$ by $H_1$, then the assertion still holds.


\begin{lemma}\label{3-subgraph}
The following three assertions hold.
\vspace{-8pt}
\begin{enumerate}[(\romannumeral1)]
\setlength{\itemsep}{-1ex}
\item  If $F_1\subseteq G$, then for each vertex $v_k$ in $V_2$, $N(v_k)=\{v_{k-1},v_{k+1},u\}$.
\item  If $F_2\subseteq G$, then $N(v_2)=\{v_1,v_3,u\}$.
\item  If $F_3\subseteq G$, then $N(v_{s+1})=\{v_s,v_{s+2},u\}$.
\end{enumerate}
\vspace{-8pt}
Furthermore, $F_1\subseteq G$, $F_2\varsubsetneq G$, $F_3\varsubsetneq G$, $N(w_{t_1})=\{w_{t_1-1},v_t,u\}$, and $N(w_1)=\{w_2,v_1,u\}$.
\end{lemma}

\begin{proof}
Let
\begin{equation}
C_1=\begin{cases}
uv_1w_1P_1w_{t_1}v_tu, & if\ F_1\subseteq G\\
w_pP_1w_{t_1}w_p, & if\ F_2\subseteq G\\
v_1Pv_sw_{t_1}P_1w_1v_1, & if\ F_3\subseteq G\\
\end{cases}
\nonumber
\end{equation}
and
\begin{equation}
v^*=\begin{cases}
v_k, & if\ F_1\subseteq G\\
v_2, & if\ F_2\subseteq G\\
v_{s+1}, & if\ F_3\subseteq G,\\
\end{cases}
\nonumber
\end{equation}
where $v_k$ is  a vertex  in $V_2$. Then $C_1$ is an odd cycle in $G$ and $v^*\in V_2$.

\begin{claim}\label{claim31}
No vertex of  $P$ is  reachable from the vertex $v^*$ by an open $M_1'\backslash E(P_1)$-alternating path in $G$ unless the path is an edge in $E(P)$ with one end $v^*$.
\end{claim}

Suppose, to the contrary, that there exists a vertex $v_q$ in $P$  which is reachable from $v^*$ by an open $M_1'\backslash E(P_1)$-alternating path $Q$ in $G$ and $Q$ is not an  edge $v_qv^*$ in $E(P)$. 
Then $v_qv^*\notin E(P)\cup E(P_1)$ when $|E(Q)|=1$.
Further, by the definition, all internal vertices of the path $Q$ lie in the graph $H_1$.
It follows that whether $|E(Q)|=1$ or not, we have $E(Q)\cap (E(P)\cup E(P_1))=\emptyset$.
Note that $Q$ is also an open $M_1'$-alternating path.
By Lemma \ref{lemma32}, we have $v_q\in V_2$ since $v^*\in V_2$.
It is possible  that $v_q=v^*$.
Let $C_2=v^*Qv^*$ if $v_q=v^*$, and  $C_2=v^*Pv_qQv^*$ otherwise.
Then $C_2$ is an odd cycle in $G$.
 For the case  $F_3\subseteq G$ and $q<s$, $v_1Pv_qQv^*v_sw_{t_1}P_1w_1v_1$ is an $M_2$-alternating cycle of $G-\{u,v_t\}$, contradicting Lemma \ref{lemma31}.
For all the remaining cases, one can obtain that $C_1$ and $C_2$ are two vertex disjoint odd cycles of $G$ such that $G-(V(C_1)\cup V(C_2))$ has a perfect matching, contradicting the hypothesis that $G$ is solid.
The claim holds.

\begin{claim}\label{claim32}
No vertex of  $P_1$ is reachable from the vertex $v^*$ by an open $M_1'\backslash E(P_1)$-alternating path in $G$.
\end{claim}

Suppose, to the contrary, that there exists a vertex $w_j$ in $P_1$ such that $w_j$ and $v^*$ are connected by an open $M_1'\backslash (E(P_1)$-alternating path $Q'$ in $G$, where $1\leq j\leq t_1$.
Then all internal vertices of the path $Q'$ lie in the graph $H_1$.
If $j$ is even, then no matter which of $F_1$, $F_2$ and $F_3$ is a subgraph of  $G$, we see  that
$v_1Pv^*Q'w_jP_1w_1v_1$ is an $M_2$-alternating cycle of $G-\{u,v_t\}$. This contradicts Lemma \ref{lemma31}.
Now we consider the case where $j$ is odd.
If $F_1\subseteq G$, then $v^*Pv_tw_{t_1}P_1w_jQ' v^*$ is an $M_1$-alternating cycle of $G-\{u,v_1\}$;
 if $F_3\subseteq G$, then $v_sv^*Q'w_jP_1w_{t_1}v_s$ is an $M_2$-alternating cycle of $G-\{u,v_t\}$. In both cases, we get a  contradiction to Lemma \ref{lemma31}.
So we may assume that $F_2\subseteq G$.
Recall that $j$ is odd, and both $p$ and $t_1$ are even. Then $j\neq p$ and $j\neq t_1$.
If $p<j<t_1$, then $v_1Pv^*Q'w_jP_1w_{t_1}w_pP_1w_1v_1$ is an $M_2$-alternating cycle of $G-\{u,v_t\}$, a contradiction to Lemma \ref{lemma31}.
If $1\leq j<p$, then let $C_3=v_1Pv^*Q'w_jP_1w_1v_1$.
Note that  $C_1$ and $C_3$ are two vertex disjoint odd cycles of $G$ such that $G-(V(C_1)\cup V(C_3))$ has a perfect matching, contradicting the hypothesis that $G$ is solid.
The claim holds.

\vspace{0.25cm}

Next, we assert that $v^*$  has no neighbour in $H_1$.
Recall that $H_1=G-(V(P)\cup V(P_1)\cup\{u\})$.
Assume, to the contrary, that $v^*$  has a neighbour $a_1$ in $H_1$.
Let $P'=a_1a_2\cdots a_l$ be a maximal $M_1$-alternating path  in $H_1$ starting at $a_1$ such that $a_1a_2\in M_1$.
Then $a_{l-1}a_l\in M_1$ and the vertex $a_l$ has no neighbour in $H_1-V(P')$ by the choice of $P'$.
By Claim \ref{claim31} and \ref{claim32}, the vertex $a_l$ has no neighbour in $P\cup P_1$.
Therefore, $P'$ is a maximal $M_1$-alternating path in $H$ starting at $a_1$ with $a_1a_2\in M_1$.
Now we consider the case that the vertex $a_l$ has a neighbour $a_{l'}$ in $P'$, where $1\leq l'\leq l-2$.
By Lemma \ref{lemma33}, $l'$ is even.
Let $C_4=a_{l'}P'a_la_{l'}$.
Then $C_1$ and $C_4$ are two vertex disjoint odd cycles of $G$ such that $G-(V(C_1)\cup V(C_4))$ has a perfect matching, a contradiction.
It follows that the vertex $a_l$ has no neighbour  in $P'$ besides the vertex $a_{l-1}$.
Consequently, $N(a_l)=\{a_{l-1}\}$ or $N(a_l)=\{a_{l-1},u\}$.
We get a contradiction to the fact that $\delta(G)\geq3$.
This implies that $v^*$  has no neighbour in $H_1$ and so the assertion holds.

Again, by Claims \ref{claim31} and \ref{claim32}, the vertex $v^*$  has no neighbour in $P\cup P_1$ besides its two neighbours in $P$.
This implies that the vertex $v^*$  has precisely one additional neighbor $u$ in $G$ besides the  two  in $P$, as required by the above assertion and  $\delta(G)\geq3$.
Thus, (i)-(iii) hold.

To complete the proof of Lemma \ref{3-subgraph}, we proceed to show that $F_1\subseteq G$.
If $F_2\subseteq G$ or  $F_3\subseteq G$, let $C_5=uv^*Pv_tu$.
Then $C_1$ and $C_5$ are two vertex disjoint odd cycles of $G$ such that $G-(V(C_1)\cup V(C_5))$ has a perfect matching, a contradiction.
Thus, we have $F_2\varsubsetneq G$ and $F_3\varsubsetneq G$.

Since $\delta(G)\geq3$,  $v_1$ has at least one neighbour in $H$ or $P-\{v_1,v_2\}$.
If $v_1$ has a neighbour $z$ in $P-\{v_1,v_2\}$, then, by Lemma \ref{lemma32}, we have $z\in V_1$.
This implies that $z=v_t$ because $F_3\varsubsetneq G$. Hence, $F_1\subseteq G$.
If $v_1$ has a neighbour $w_1$ in $H$, let  $P_1=w_1w_2\cdots w_{t_1}$ be defined as above (after the proof of Lemma \ref{lemma33}),
then $w_{t_1}$ has no neighbour in $H_1$ by the choice of $P_1$.
By Lemma \ref{lemma33}, we have $N(w_{t_1})\cap V(P_1)\subseteq \{w_2,w_4,\ldots,w_{t_1-2},w_{t_1-1}\}$.
It follows that $N(w_{t_1})\cap V(P_1)=\{w_{t_1-1}\}$ because $F_2\varsubsetneq G$.
Since $\delta(G)\geq3$, $w_{t_1}$  has a neighbour $z'$ in $P$.
By Lemma \ref{lemma32}, we have $z'\in V_1$.
This implies that $z'=v_t$ because $F_3\varsubsetneq G$.
Once more invoking the minimum degree condition  $\delta(G)\geq3$, we have $N(w_{t_1})=\{w_{t_1-1},v_t,u\}$, which implies that  $F_1\subseteq G$.
By symmetry, it follows that  $N(w_1)=\{w_2,v_1,u\}$.
The result holds.
\end{proof}

By Lemma \ref{3-subgraph}, we have $F_1\subseteq G$ and $N(v_k)=\{v_{k-1},v_{k+1},u\}$ for any vertex $v_k$ in $V_2$.
If $V(P_1)\neq \emptyset$, then  $N(w_{t_1})=\{w_{t_1-1},v_t,u\}$ and let $C=v_1Pv_tw_{t_1}P_1w_1v_1$;
otherwise, $v_1v_t\in E(G)$ and let $C=v_1Pv_tv_1$.
Then $C$ is an odd cycle of $G$.
In the rest of this paper, we show that $u$ is a neighbour of  every vertex  in $V_1\cup V(P_1)$.

\begin{lemma}\label{v1-p1}
For any vertex $y\in V_1\cup V(P_1)$, if no vertex of  $C$ is reachable from $y$ by an open $M_1'\backslash  E(P_1)$-alternating path in $G$ unless the path is an edge  in $E(C)$ incident with  $y$,
then $d(y)=3$ and  besides  its two neighbours in $C$, $y$ has exactly one more neighbour $u$.
\end{lemma}

\begin{proof}
Since  $y\in V_1\cup V(P_1)$, $y$ lies in $C$.   Recall that each edge of $G$ not in $M_1'\backslash  E(P_1)$ is an open $M_1'\backslash  E(P_1)$-alternating path in $G$.
Assume that  no vertex of  $C$ is reachable from $y$ by an open $M_1'\backslash  E(P_1)$-alternating path in $G$ unless the path is an edge  in $E(C)$ incident with  $y$,
Then  $y$ has no more neighbour in $C$ other than the two neighbours.

Recall that $H_1=G-(V(P)\cup V(P_1)\cup\{u\})$ and $M_1\cap E(H_1)$ is the unique perfect matching of $H_1$.
If $y$ has a neighbour $y_1$ in $H_1$, let $P''=y_1y_2\cdots y_l$ be
a maximal $M_1$-alternating path in $H_1$ starting at $y_1$ with $y_1y_2\in M_1$.
Then  $y_{l-1}y_l\in M_1$ and the vertex $y_l$ has no neighbour in $H_1-V(P'')$ by the choice of  $P''$.
Moreover, by the above  assumption, the vertex $y_l$  has no neighbour in $C$.
Consequently, the neighbours of $y_l$ belong to $V(P'')\cup\{u\}$.

Next, we show that $N(y_l)\cap V(P'')=\{y_{l-1}\}$.
Suppose, to the contrary,  that $y_ly_j\in E(G)$, where $1\leq j\leq l-2$. Let $C_1=y_lP''y_jy_l$.
If $j$ is odd, then $C_1$ is an $M_1$-alternating cycle of $G-\{u,v_1\}$,  contradicting Lemma \ref{lemma31}.
If $j$ is even, then $C_1$ is an odd cycle in $G$.
By Lemma \ref{3-subgraph}, we have $uv_2\in E(G)$ and $uv_{t-1}\in E(G)$.
If $y\in (V_1\backslash\{v_t\})\cup\{w_2,w_4,\ldots,w_{t_1}\}$, let $C_2=uv_{t-1}v_tu$;
otherwise, let $C_2=uv_1v_2u$.
Then $C_1$ and $C_2$ are two vertex disjoint odd cycles of $G$ such that $G-(V(C_1)\cup V(C_2))$ has a perfect matching, a contradiction.
So $N(y_l)\cap V(P'')=\{y_{l-1}\}$.

From the preceding discussion, we see that  $N(y_l)=\{y_{l-1}\}$ or $N(y_l)=\{y_{l-1},u\}$, contradicting the fact that $\delta(G)\geq3$.
Therefore, the vertex $y$ has no neighbour in $H_1$.
It follows that $y$ has no  neighbour in $C\cup H_1$ other than the two neighbours in $C$.
Since $\delta(G)\geq3$, $y$ has exactly one more neighbour $u$ in $G$ other than the  two neighbours in $C$, and so $d(y)=3$.
The result holds.
\end{proof}

 The following two lemmas state that each vertex in $V_1\cup V(P_1)$ has degree  three in $G$ and is adjacent to the vertex $u$.

\begin{lemma}\label{v1}
For each vertex $v_i$ in $V_1\backslash\{v_1,v_t\}$,  $N(v_i)=\{v_{i-1},v_{i+1},u\}$.
\end{lemma}

\begin{proof}
If we can show the assertion that no vertex of  $C$ is reachable  from $v_i$ by an open $M_1'\backslash  E(P_1)$-alternating path in $G$ unless the path is an  edge in $E(P)$ incident with the vertex $v_i$, then  the result holds from Lemma \ref{v1-p1}.
We now establish   this assertion by contradiction.
Recall that $V(C)=V(P_1)\cup V(P)$. By Lemma \ref{3-subgraph},  $uv_2\in E(G)$ and $uv_{t-1}\in E(G)$.

If there exists a vertex $v_q$ in $P$ such that $v_q$ is reachable  from $v_i$
by an open $M_1'\backslash  E(P_1)$-alternating path $Q$ in $G$ besides the path is an edge $v_qv_i$ in $E(P)$,
then $E(Q)\cap E(C)=\emptyset$ and all internal vertices of the path $Q$ lie in the graph $H_1$.
By Lemma \ref{3-subgraph}(i), $v_q\in V_1$.
It  is possible that $v_q=v_i$.
If $v_q=v_i$, let $C_1=v_iQv_i$; otherwise, let $C_1=v_iPv_qQv_i$.
If $q<t$,  let $C_2=uv_{t-1}v_tu$; otherwise, let $C_2=uv_1v_2u$.
Then $C_1$ and $C_2$ are two vertex disjoint odd cycles of $G$ such that $G-(V(C_1)\cup V(C_2))$ has a perfect matching, a contradiction.

If there exists a vertex $w_j$ in $P_1$ such that $w_j$ and $v_i$ are connected
by an open $M_1'\backslash  E(P_1)$-alternating path $Q'$ in $G$, where $1\leq j\leq t_1$,
then $E(Q')\cap E(C)=\emptyset$ and all internal vertices of the path $Q'$ lie in the graph $H_1$.
If $j$ is even, let $C_3=v_1Pv_iQ'w_jP_1w_1v_1$  and $C_4=uv_{t-1}v_tu$; otherwise, let $C_3=v_iPv_tw_{t_1}P_1w_jQ'v_i$ and $C_4=uv_1v_2u$.
Then $C_3$ and $C_4$ are two vertex disjoint odd cycles of $G$ such that $G-(V(C_3)\cup V(C_4))$ has a perfect matching, a contradiction.
Hence, the assertion  holds.
\end{proof}

\begin{lemma}\label{p1}
Each vertex in $V(P_1)\cup\{v_1,v_t\}$ has exactly one more neighbour $u$ in $G$ other than the  two neighbours in $C$.
\end{lemma}

\begin{proof}
By Lemma \ref{3-subgraph}, the result holds for the two vertices $w_1$ and $w_{t_1}$ in $P_1$.
Recall that $t_1$ is even. For convenience, let $w_0=v_1$, $w_{t_1+1}=v_t$, and $P_1'=w_0w_1P_1w_{t_1}w_{t_1+1}$.
Let  $w_j$ be a vertex of  $P_1'$. We first consider the case that $j\in \{0, 2, 4, \ldots, t_1-2\}$.
Then $j$ is even.

We assert that no vertex of $P_1'$ is reachable  from $w_j$ by an open $M_1'\backslash  E(P_1)$-alternating path in $G$  unless the path is an edge in  $E(P_1')$ incident with $w_j$.
Suppose,  to the contrary, that there exists a vertex $w_k$ in $P_1'$ such that $w_k$ is reachable  from $w_j$ by an open $M_1'\backslash  E(P_1)$-alternating path $Q$ in $G$ and $Q$ is not an edge in  $E(P_1')$ incident with $w_j$.
Then $E(Q)\cap E(C)=\emptyset$ and all internal vertices of the path $Q$ lie in the graph $H_1$.
It  is possible  that $w_k=w_j$.
If $w_k=w_j$, let $C_1=w_jQw_j$; otherwise, let $C_1=w_jP_1'w_kQw_j$.
Let $C_2=uv_{t-1}v_tu$.
If $k$ is even, then $C_1$ and $C_2$ are two vertex disjoint odd cycles of $G$ such that $G-(V(C_1)\cup V(C_2))$ has a perfect matching, a contradiction.

So we assume that $k$ is odd.
Then $w_k\neq w_j$ and $C_1$ is an even cycle in  $G$.
If $k<j$, then $C_1$ is an $M_1$-alternating cycle of $G-\{u,v_1\}$,  contradicting Lemma \ref{lemma31}.
Thus  $k>j$.
Note that $G-(\{v_2,u\}\cup V(C_1))$ has a perfect matching.
Then the edge $v_2u$ is nonsolitary in $G$.
If $V(P_1)=\emptyset$, then $w_j=w_0=v_1$ and $w_k=w_{t_1+1}=v_t$.
It follows that the set $V(Q)\backslash\{v_1,v_t\}\neq\emptyset$ since $G$ is simple.
Therefore, we may assume that $V(P_1)\neq\emptyset$ by swapping $P_1'$ and $Q$ if necessary.
By Lemma \ref{3-subgraph}, we have $uw_1\in E(G)$ and $uw_{t_1}\in E(G)$.
Let $C_3=uw_1P_1'v_tu$ and $C_4=uw_{t_1}P_1'v_1u$.
Then both $C_3$ and $C_4$ are even in $G$.
Note that both  $G-(\{v_2,v_1\}\cup V(C_3))$ and $G-(\{v_2,v_3\}\cup V(C_4))$ have a perfect matching,
which implies that both $v_2v_1$ and $v_2v_3$ are nonsolitary in $G$.
Therefore, the vertex $v_2$ is incident with three nonsolitary edges in $G$, contradicting Lemma \ref{2-nonsolitary}.
The assertion holds.

By Lemmas \ref{3-subgraph} and \ref{v1},  we see that each vertex in $V(P)\backslash\{v_1,v_t\}$ has only three neighbours in $G$, one of which is $u$.

By the above  assertion  and Lemma \ref{v1-p1}, $w_j$ has exactly one more  neighbour $u$ in $G$ besides the two neighbours in $C$.
By symmetry, each vertex in $\{w_3,w_5,\ldots,w_{t_1-1}, w_{t_1+1}\}$ also has exactly one more neighbour $u$ in $G$ besides the  two neighbours in $C$.
The result holds.
\end{proof}

We now proceed to complete the proof of   Theorem \ref{main-theorem}.
By Lemmas \ref{3-subgraph}, \ref{v1} and \ref{p1}, every vertex of  $C$ has exactly one more neighbour $u$
in $G$ other than the two neighbours in $C$.
This implies that if  $V(H_1)\neq \emptyset$, then $u$ is a cut vertex of $G$, contradicting the fact that $G$ is 3-connected.
Hence, $V(H_1)=\emptyset$ and so $V(G)=V(C)\cup \{u\}$.
It follows  that $G$ is a wheel $W_n$.
The proof of Theorem \ref{main-theorem} is completed.  \\

\section*{Acknowledgements}
This work is supported by the National Natural Science Foundation of China (Nos. 12171440 and 12371361)
and the Natural Science Foundation of Henan Province (No. 252300421786).

\end{document}